\documentclass{amsart}
\usepackage{amsthm,hyperref}
\usepackage{amsmath,amssymb,MnSymbol}
\usepackage{cite}
\usepackage{tikz}
\numberwithin{equation}{section}
\newtheorem{thm}{Theorem}[section]

\newtheorem{cor}[thm]{Corollary}
\newtheorem{lem}[thm]{Lemma}

\theoremstyle{definition}
\newtheorem{defn}[thm]{Definition}

\newtheorem{rem}[thm]{Remark}

\DeclareMathOperator{\Mon}{Mon}
\DeclareMathOperator{\lex}{lex}

\DeclareMathOperator{\supp}{supp}

\DeclareMathOperator{\K.dim}{K.dim}
\begin{document}
\title{\textbf{Hilbert functions of monomial ideals containing a regular sequence}}
\author{Abed Abedelfatah}
\address{Institute of Mathematics, Hebrew University of Jerusalem, Jerusalem 91904, Israel}
\email{abed@math.haifa.ac.il}
\begin{abstract}
Let $M$ be an ideal in $K[x_1,\dots,x_n]$ ($K$ is a field) generated by products of linear forms and containing a homogeneous regular sequence of some length. We prove that ideals containing $M$ satisfy the Eisenbud-Green-Harris conjecture and moreover prove that the Cohen-Macaulay property is preserved. We conclude that monomial ideals satisfy this conjecture. We obtain that $h$-vector of Cohen-Macaulay simplicial complex $\Delta$ is the $h$-vector of Cohen-Macaulay $(a_1-1,\dots,a_t-1)$-balanced simplicial complex where $t$ is the height of the Stanley-Reisner ideal of $\Delta$ and $(a_1,\dots,a_t)$ is the type of some regular sequence contained in this ideal.
\end{abstract}
\maketitle

\section{Introduction}
The Eisenbud-Green-Harris (EGH), in the general form, states that for every graded ideal $I$ of height $t$ in $K[x_1,\dots,x_n]$, where $K$ is a field, containing a regular sequence $f_1,\dots,f_t$ of degrees $a_1\leq \cdots \leq a_t$, there exists a graded ideal $J$ containing $x_1^{a_1},\dots,x_t^{a_t}$ with the same Hilbert function. In this article, we prove the conjecture when $(f_1,\dots,f_t)\subseteq M\subseteq I$, where $M$ is a totally reducible ideal (see definition \ref{1}) and moreover prove that if the original ideal $I$ is Cohen-Macaulay then $J$ can be taken as a Cohen-Macaulay ideal.
In the case that each polynomial in the regular sequence is a product of linear forms, we can take $M$ to be the complete intersection ideal $\langle f_1,\dots,f_t\rangle$. So this extend the result in \cite{abed}.

We conclude that the EGH conjecture is true when $I$ is a monomial ideal. A related result was proved in \cite{CCV} when $I$ is generated by monomials of degree 2 and in \cite{J.M} when $f_1,\dots,f_t$ is a monomial regular sequence.

In \cite{CCV}, Caviglia, Constantinescu and Varbaro proved that $h$-vectors of Cohen-Macaulay flag simplicial complexes are $h$-vectors of Cohen-Macaulay balanced simplicial complexes. We generalize by proving that if $\Delta$ is a Cohen-Macaulay simplicial complex such that $I_{\Delta}$ of height $t$ and containing a regular sequence $f_1,\dots,f_t$ with $\deg(f_i)=a_i$ for all $1\leq i\leq t$, then there exists a Cohen-Macaulay $(a_1-1,\dots,a_t-1)$-balanced simplicial complex $\Gamma$ with the same $h$-vector.

\section{Preliminaries and notations}
A proper ideal $I$ in $S=K[x_1,\dots,x_n]$ is called \emph{graded} or \emph{homogeneous} if it has a system of homogeneous generators. Let $R=S/I$, where $I$ is a homogeneous ideal. The sequence $H(R)=\{H(R,n)\}_{n\geq0}$, where $H(R,n):=\dim_{K}R_n=\dim_{K}S_n/I_n$ is called the Hilbert function of $I$ (or $R$), and $H_{R}(t)=\sum_{0\leq n\in\mathbb{Z}}H(R,n)t^n$ is the Hilbert series of $R$.
For simplicity, we denote the dimension of a $K$-vector space $V$ by $|V|$ instead of $\dim_{K}V$. For a $K$-vector space $V\subseteq S_d$, where $d\geq 0$, we denote by $S_1V$ the $K$-vector space spanned by $\{x_iv:~1\leq i\leq n~\wedge~v\in V\}$. Throughout this paper $\textbf{A}=(a_1,\dots,a_n)\in \mathbb{Z}^n$, where $a_1\leq\cdots\leq a_n$. For a subset $A$ of $S$, we denote by $\Mon(A)$ the set of all monomials in $A$. Let $f_1,\dots,f_t$ be a regular sequence in $S$ with $\deg(f_i)=a_i$ for all $i$. We say that $(a_1,\dots,a_t)$ is the \emph{type} of the homogeneous polynomials $f_1,\dots,f_t$ if $\deg(f_i)=a_i$ for all $1\leq i\leq t$.

We define the \emph{lex order} on $\Mon(S)$ by setting $\textbf{x}^b=x_1^{b_1}\cdots x_n^{b_n}<_{\lex}x_1^{c_1}\cdots x_n^{c_n}=\textbf{x}^c$ if either $\deg(\textbf{x}^b)<\deg(\textbf{x}^c)$ or $\deg(\textbf{x}^b)=\deg(\textbf{x}^c)$ and $b_i<c_i$ for the first index $i$ such that $b_i\neq c_i$. We recall the definitions of lex ideal and lex-plus-powers ideal.
A graded ideal is called \emph{monomial} if it has a system of monomial generators. A monomial ideal $I\subseteq S$ is called \emph{lex}, if whenever $I\ni z<_{\lex}w$, where $w,z$ are monomials of the same degree, then $w\in I$.
A monomial ideal $I$ is \emph{$\textbf{A}$-lex-plus-powers} if there exists a lex ideal $L$ such that $I=\langle x_1^{a_1},\dots,x_n^{a_n}\rangle+L$.

A \emph{simplicial complex} $\Delta$ on the set $A=\{1,\dots,n\}$ is a collection of subsets $A$ such that
\begin{itemize}
  \item $\{i\}\in\Delta$ for every $1\leq i\leq n$.
  \item If $F\in\Delta$ and $G\subseteq F$, then $G\in\Delta$.
\end{itemize}
Each element $F\in\Delta$ is called a \emph{face} of $\Delta$. A maximal face of $\Delta$ with respect to inclusion is called a facet and we will denote by $\mathcal{F}(\Delta)$ the set of facets of $\Delta$. The dimension of the face $F$ is $|F|-1$ and the \emph{dimension of $\Delta$} is $\max\{\dim F~:~F\in\Delta\}$. We denote by $f_i=f_i(\Delta)$ the number of faces of $\Delta$ of dimension $i$. The sequence $f(\Delta)=(f_{-1},f_0,\dots,f_{d-1})$, where $d-1$ is the dimension of $\Delta$, is called the \emph{$f$-vector} of $\Delta$. The Stanley-Reisner ideal of $\Delta$ is the ideal $I_{\Delta}$ of $S$ generated by $x_F=x_{i_1}\cdots x_{i_j}$ where $1\leq i_1<\cdots< i_j\leq n$ and  $F=\{i_1,\dots,i_j\}\notin\Delta$. The Stanley-Reisner ring is defined by $K[\Delta]=S/I_{\Delta}$.  If $d=\K.dim(K[\Delta])$, then $H_{K[\Delta]}(t)=\frac{h_0+h_1t+\dots+h_dt^d}{(1-t)^d}$, where $h_i\in \mathbb{Z}$ for all $i$. The sequence  $h(\Delta)=(h_0,\dots,h_d)$ is called the \emph{$h$-vector of $\Delta$}. The $f$-vector and the $h$-vector of a $(d-1)$-dimensional simplicial complex $\Delta$ are related by $$\sum_{i=0}^{d}h_it^{d-i}=\sum_{i=0}^{d}f_{i-1}(t-1)^{d-i}.$$
A simplicial complex is called \emph{flag} if all its minimal nonfaces have cardinality two, i.e. the Stanley-Reisner ideal is generated by square-free monomials of degree two. A simplicial complex $\Delta$ is called Cohen-Macaulay (CM) over a field $K$ if $K[\Delta]$ is Cohen-Macaulay.\\
Let $b_1,\dots,b_r$ be a positive integers. A simplicial complex $\Delta$ is called \emph{$(b_1,\dots,b_r)$-balanced} if $1+\dim \Delta=\sum_{k=1}^rb_k$ and the vertex set of $\Delta$ can be partitioned into $r$ sets $V_1,\dots,V_r$ such that $|F\cap V_k|\leq b_k$ for every face $F$ of $\Delta$ and all $1\leq k\leq r$. If $b_i=1$ for all $i$, then $\Delta$ is called \emph{balanced complex}.

\section{The main results}

\begin{defn}\label{1}
A graded ideal in $S$ is called \emph{totally reducible} if it is generated by products of linear forms.
\end{defn}
\begin{lem}\label{20}
Let $K$ be an infinite field and $I$ be an ideal of $S=K[x_1,\dots,x_n]$. Assume that $I$ is a totally reducible ideal that contains a regular sequence $g_1,\dots,g_t$ with $\deg(g_i)=a_i$. If $f_1=\ell_1\cdots\ell_{a_1}$ is some polynomial in $I$ of degree $a_1$, where $\ell_i\in S_1$ for all $i$, then $I$ contains a regular sequence $f_1,f_2,\dots,f_t$ such that $\deg(f_i)=a_i$ for all $1\leq i\leq t$.
\end{lem}

\begin{proof}
Suppose that we found $f_j$, where $1\leq j<t$. Let $P_1,\dots,P_r$ be the minimal prime ideals over $F=\langle f_1,\dots,f_j\rangle$. By [\citen{sm}, Lemma 4.1], $I_{a_{j+1}}$ contains a regular sequence $h_1,\dots,h_{j+1}$ such that $\deg(h_i)=a_{j+1}$ for all $1\leq i\leq j+1$. So $I_{a_{j+1}}\nsubseteq P_i$ for all $1\leq i\leq r$. Since the field $K$ is infinite, $I_{a_{j+1}}\nsubseteq P_1\cup\cdots\cup P_r$. It follows that there is $f_{j+1}\in I_{a_{j+1}}$ that avoids every $P_i$. Note that $S$ is CM and $F$ is generated by $\mathrm{ht}(F)$ elements. So $F$ is unmixed and it follows that $f_{j+1}$ is a non-zero-divisor in $S/F$.
\end{proof}

\begin{lem}\label{23}
Let $I$ be a graded ideal in $S$ containing a regular sequence $f_1,\dots,f_t$ with $\deg(f_i)=a_i$. Assume that for every $d\geq 0$, there is an ideal $L_d$ containing $\langle x_1^{a_1},\dots,x_t^{a_t}\rangle_d$ and $\langle x_1^{a_1},\dots,x_t^{a_t}\rangle_{d+1}$ such that $H(S/L_d,d)=H(S/I,d)$ and $H(S/L_d,d+1)\geq H(S/I,d+1)$. Then $I$ has the same Hilbert function as an ideal containing $x_1^{a_1},\dots,x_t^{a_t}$.
\end{lem}

\begin{proof}
By Clements-Lindstr{\"o}m's theorem \cite{clements}, we may assume that $L_{d,d}$ and $L_{d,d+1}$ are lex vector spaces in $S/\langle x_1^{a_1},\dots,x_t^{a_t}\rangle$ for all $d\geq0$, where $L_{d,i}$ is the $i$-th component of $L_d$. Let $L=\oplus_{d\geq0}L_{d,d}$. Since $|L_{d,d+1}|\leq|I_{d+1}|=|L_{d+1,d+1}|$, it follows that $L_{d,d+1}\subseteq L_{d+1,d+1}$, for all $d$. So $S_1L_{d,d}\subseteq L_{d,d+1}\subseteq L_{d+1,d+1}$, for all $d$. Thus, $L$ is an ideal. Clearly, $H(S/I)=H(S/L)$ and $x_1^{a_1},\dots,x_t^{a_t}\in L$.
\end{proof}

\begin{thm}\label{21}
Assume that $M$ is a totally reducible ideal that contains a regular sequence $f_1,\dots,f_t$ of degrees $\deg(f_i)=a_i$. If $I$ is a graded ideal in $S$ containing $M$, then $I$ has the same Hilbert function as an ideal containing $x_1^{a_1},\dots,x_t^{a_t}$.
\end{thm}

\begin{proof}
Since the Hilbert function of a monomial ideal is independent of the filed $K$, by an extension of $K$ we may assume that $K$ is infinite. We prove the theorem by induction on $n$. Let $n>1$. By lemma \ref{20}, we may assume that $f_1=\ell_1\cdots \ell_r$, where $\ell_1,\dots,\ell_r\in S_1$ ($r=a_1)$. Let $d\geq0$. By lemma \ref{23}, we need to find a graded ideal $L$ containing $\langle x_1^{a_1},\dots,x_t^{a_t}\rangle_d$ and $\langle x_1^{a_1},\dots,x_t^{a_t}\rangle_{d+1}$ such that $H(S/L,d)=H(S/I,d)$ and $H(S/L,d+1)\geq H(S/I,d+1)$. Without loss of generality, we may assume that

\begin{center} $|I_d\cap \langle \ell_1\rangle_d|\geq |I_d\cap \langle \ell_i\rangle_d|$ for all $2\leq i\leq r$,\\
$|(I:\ell_1)_{d-1}\cap \langle \ell_2\rangle_{d-1}|\geq |(I:\ell_1)_{d-1}\cap \langle \ell_i\rangle_{d-1}|$ for all $3\leq i\leq r$,\\
$|(I:\ell_1\ell_2)_{d-2}\cap \langle \ell_3\rangle_{d-2}|\geq |(I:\ell_1\ell_2)_{d-2}\cap \langle \ell_i\rangle_{d-2}|$ for all $4\leq i\leq r$,\\
$\vdots$\\
$|(I:\ell_1\cdots \ell_{r-2})_{d-(r-2)}\cap \langle \ell_{r-1}\rangle_{d-(r-2)}|\geq |(I:\ell_1\cdots \ell_{r-2})_{d-(r-2)}\cap \langle \ell_r\rangle_{d-(r-2)}|$.\end{center}

By considering the short exact sequences
\begin{center}
$0\rightarrow S/(I:\ell_1){\longrightarrow}S/I{\longrightarrow}S/I+\langle \ell_1\rangle\rightarrow 0$,\\
$0\rightarrow S/(I:\ell_1\ell_2){\longrightarrow}S/(I:\ell_1){\longrightarrow}S/(I:\ell_1)+\langle \ell_2\rangle\rightarrow 0$,\\
$0\rightarrow S/(I:\ell_1\ell_2\ell_3){\longrightarrow}S/(I:\ell_1\ell_2){\longrightarrow}S/(I:\ell_1\ell_2)+\langle \ell_3\rangle\rightarrow 0$,\\
$\vdots$\\
$0\rightarrow S/(I:\ell_1\cdots \ell_{r-1}){\longrightarrow}S/(I:\ell_1\cdots \ell_{r-2}){\longrightarrow}S/(I:\ell_1\cdots \ell_{r-2})+\langle \ell_{r-1}\rangle\rightarrow 0$.
\end{center}
we see that $H(S/I,t)$ is equal to $$H(S/I+\langle \ell_1\rangle,t)+\sum_{i=1}^{r-2}H(S/(I:\ell_1\cdots \ell_i)+\langle \ell_{i+1}\rangle,t-i)+H(S/(I:\ell_1\cdots \ell_{r-1}),t-(r-1))$$ for all $t\geq0$.\\

Let $J_0=I+\langle \ell_1\rangle$, $J_{r-1}=(I:\ell_1\cdots \ell_{r-1})$, and for $1\leq i\leq r-2$ let $J_i=(I:\ell_1\cdots \ell_i)+\langle \ell_{i+1}\rangle$. Note that for all $0\leq i\leq r-1$ and $H(\frac{S/\langle \ell_{i+1}\rangle}{J_i/\langle \ell_{i+1}\rangle})=H(S/J_i)$. For all $0\leq i\leq r-1$, let $M_i$ be the image of $M$ in $S/\langle \ell_{i+1}\rangle$. It is clear that $M_i\subseteq J_i/\langle \ell_{i+1}\rangle$ is a totally reducible ideal that contains a regular sequence $\overline{f_2},\dots,\overline{f_t}$. For all $0\leq i\leq r-1$, $S/\langle \ell_{i+1}\rangle$ is isomorphic to $\overline{S}=K[x_2,\dots,x_n]$, so by the inductive step there is an ideal in $\overline{S}$ containing $x_2^{a_2},\dots,x_{t}^{a_{t}}$ with the same Hilbert function as $J_{i}$. For all $0\leq i\leq r-1$, let $L_i$ be the lex-plus-powers ideal in $\overline{S}$ containing $x_2^{a_2},\dots,x_{t}^{a_{t}}$ such that $H(\overline{S}/L_i)=H(S/J_i)$.\\\\

\textbf{\emph{Claim:}} $L_{i,d-i}\subseteq L_{i+1,d-i}$ for all $0\leq i\leq r-2$, where $L_{i,j}$ is the $j$-th component of the ideal $L_i$.\\\\
\emph{Proof of the claim:} If $i=0$, then
\begin{align*}
|J_{0,d}|&=|I_d|+|\langle \ell_1\rangle_d|-|I_d\cap \langle \ell_1\rangle_d|\\
&\leq|I_d|+|\langle \ell_1\rangle_d|-|I_d\cap \langle \ell_2\rangle_d|\\
&=|I_d|+|\langle \ell_2\rangle_d|-|I_d\cap \langle \ell_2\rangle_d|\\
&=|I_d+\langle \ell_2\rangle_d|\\
&\leq |J_{1,d}|.
\end{align*}
So $H(\overline{S}/L_0,d)\geq H(\overline{S}/L_1,d)$. Since $L_0$ and $L_1$ are lex-plus-powers ideals, it follows that $L_{0,d}\subseteq L_{1,d}$.
If $1\leq i\leq r-2$, then
\begin{align*}
|J_{i,d-i}|&=|(I:\ell_1\cdots \ell_i)_{d-i}|+|\langle \ell_{i+1}\rangle_{d-i}|-|(I:\ell_1\cdots \ell_i)_{d-i}\cap \langle \ell_{i+1}\rangle_{d-i}|\\
&\leq|(I:\ell_1\cdots \ell_i)_{d-i}|+|\langle \ell_{i+1}\rangle_{d-i}|-|(I:\ell_1\cdots \ell_i)_{d-i}\cap \langle \ell_{i+2}\rangle_{d-i}|\\
&=|(I:\ell_1\cdots \ell_i)_{d-i}|+|\langle \ell_{i+2}\rangle_{d-i}|-|(I:\ell_1\cdots \ell_i)_{d-i}\cap \langle \ell_{i+2}\rangle_{d-i}|\\
&=|(I:\ell_1\cdots \ell_i)_{d-i}+\langle \ell_{i+2}\rangle_{d-i}|\\
&\leq |J_{i+1,d-i}|.
\end{align*}
Similarly, we conclude that $L_{i,d-i}\subseteq L_{i+1,d-i}$, and prove the claim.\\\\
Let $K_{r}=\{zx_1^{r+j}:~z\in\Mon(\overline{S})\wedge j\geq0\}$ and $$K_{i}=\{zx_1^{i}:~z\in \Mon(L_{i})~\wedge~ d-i\leq \deg(z)\leq d-i+1\}$$ for all $0\leq i\leq r-1$. Define $L$ to be the ideal generated by $\bigcup_{0\leq i\leq r}K_i$. Since $\langle x_2^{a_2},\dots,x_t^{a_t}\rangle_{j}\subseteq K_0$, for $d\leq j\leq d+1$, it follows that $\langle x_1^{a_1},\dots,x_t^{a_t}\rangle_j\subseteq L$ for $d\leq j\leq d+1$.\\\\

\textbf{\emph{Claim:}} If $w$ is a monomial in $L$ of degree $d$ or $d+1$, then $w\in \bigcup_{0\leq i\leq r}K_i$.\\\\
\emph{Proof of the claim:} There exists a monomial $u$ in $\bigcup_{0\leq i\leq r}K_i$ such that $u|w$; i.e., $w=vu$ for some monomial $v\in S$. If $u\in K_r$, then $w\in K_r$. Assume that $u=zx_1^{i}\in K_i$, where $z\in L_i$ for some $0\leq i\leq r-1$ of degree $d-i$ or $d-i+1$. If $\deg(z)=d-i+1$, then $v=1$ and so $w\in K_i$. Similarly, if $\deg(z)=d-i$ and $\deg(w)=d$, we obtain that $w\in K_i$. Assume that $\deg(z)=d-i$ and $\deg(w)=d+1$. So $v=x_j$ for some $1\leq j\leq n$. If $j\neq 1$, then $zv\in L_{i,d-i+1}$, since $L_i$ is an ideal in $\overline{S}$. So $w\in K_i$. Assume that $j=1$. If $i+1\geq r$, then $w\in K_r$. Assume that $i+1<r$. Since $L_{i,d-i}\subseteq L_{i+1,d-i}$, it follows that $z\in L_{i+1,d-i}$. So $w=zx_1^{i+1}\in K_{i+1}$. Hence, we proved the claim.\\\\
Assume that $j=d$ or $j=d+1$. We conclude that $|L_j|=\sum_{i=0}^{r-1}|L_{i,j-i}|+\sum_{i=0}^{j-r}|\overline{S}_i|$. Since $|S_j|=\sum_{0\leq i\leq j}|\overline{S}_i|$, it follows that $$|S_j|-|L_j|=\sum_{i=j-(r-1)}^{j}|\overline{S}_{i}|-\sum_{i=0}^{r-1}|L_{i,j-i}|=\sum_{i=0}^{r-1}|\overline{S}_{j-i}|-\sum_{i=0}^{r-1}|L_{i,j-i}|.$$ So $H(S/L,j)=\sum_{i=0}^{r-1}H(\overline{S}/L_i,j-i)=\sum_{i=0}^{r-1}H(S/J_i,j-i)=H(S/I,j).$
It follows that \begin{center} $H(S/L,d)=H(S/I,d)$ and $H(S/L,d+1)=H(S/I,d+1)$.\end{center}
\end{proof}

\begin{cor}\label{22}
If $M$ is a monomial ideal containing a regular sequence $f_1,\dots,f_t$ with $\deg(f_i)=a_i$, then $M$ has the same Hilbert function as a monomial ideal in $S$ containing $x_1^{a_1},\dots,x_t^{a_t}$.
\end{cor}

\begin{rem}
It is easy to prove the case $t=n$ in corollary \ref{22} without using theorem \ref{21}:\\
If $f_1,\dots,f_n$ is a regular sequence in $M$ and $x_j^2\notin \supp(f_i)$ for some $j$ and all $1\leq i\leq n$, then $P=(f_1,\dots,f_n)\subseteq (x_1,\dots,x_{j-1},x_{j+1},\dots,x_n)$. It follows that $\mathrm{ht}(P)\leq n-1$, a contradiction. So for every $1\leq j\leq n$ there is $1\leq i\leq n$ such that $x_j^2\in \supp(f_i)$. Since $\supp(f_i)\subset M$ for all $i$, it follows that $x_1^2,\dots,x_n^2\in M$.
\end{rem}

\begin{cor}\label{25}
If $I$ is a monomial ideal in $S=K[x_1,\dots,x_n]$ of height $t$, then $I$ has the same Hilbert function as an ideal containing $x_1^{a_1},\dots,x_t^{a_t}$, for some $a_1,\dots,a_t$ satisfying $$\min\{j~|~\beta_{1,j}^{S}(S/I)\neq0\}\leq a_1\leq a_2\cdots\leq a_t\leq \max\{j~|~\beta_{1,j}^{S}(S/I)\neq0\}.$$
\end{cor}
The following theorem shows that the CM property is preserved.

\begin{thm}\label{24}
Assume that $M$ is a totally reducible ideal that contains a regular sequence $f_1,\dots,f_t$ of degrees $\deg(f_i)=a_i$. If $I$ is a CM graded ideal of height $t$ in $S$ containing $M$, then $I$ has the same Hilbert function as a CM ideal containing $x_1^{a_1},\dots,x_t^{a_t}$.
\end{thm}

\begin{proof}
We may assume that $K$ is infinite. By a linear change of variables if necessary and [\citen{cmr}, Proposition 1.5.12], we may assume that $x_{t+1},\dots,x_{n}$ is a regular sequence in $S/I$. By [\citen{cmr}, Corollary 4.1.8] there exists a unique $Q_{S/I}(r)\in \mathbb{Z}[r]$ such that $H_{S/I}(r)=\frac{Q_{S/I}(r)}{(1-r)^{n-t}}$. If $\overline{I}$ is the image of $I$ in $\overline{S}=S/(x_{t+1},\dots,x_n)$, then $H_{\overline{S}/\overline{I}}(r)=Q_{S/I}(r)$. By theorem \ref{21}, there is an ideal $\overline{L}$ in $\overline{S}$ containing $x_1^{a_1},\dots,x_t^{a_t}$ such that $H_{\overline{S}/\overline{I}}(r)=H_{\overline{S}/\overline{L}}(r)$. The $K$-algebra $\overline{S}/\overline{L}$ is of dimension zero so it is CM. Let $L=\overline{L}S$. Since $Q_{S/L}(r)=H_{\overline{S}/\overline{L}}(r)$, it follows that $Q_{S/L}(r)=Q_{S/I}(r)$. Also we have $\K.dim(S/L)=\K.dim(S/I)$, so $H_{S/L}(r)=H_{S/I}(r)$.
\end{proof}

\begin{thm}\label{25}
Let $\Delta$ be a CM simplicial complex such that $I_{\Delta}$ is of height $t$ and containing a regular sequence $f_1,\dots,f_t$ of type $(a_1,\dots,a_t)$. Then there exists a CM $(a_1-1,\dots,a_t-1)$-balanced simplicial complex $\Gamma$ with the same $h$-vector as $\Delta$.
\end{thm}

\begin{proof}
By the proof of theorem \ref{24}, there is a CM ideal  $L$ in $R=K[x_1,\dots,x_t]$ containing $x_1^{a_1},\dots,x_t^{a_t}$ such that $H(R/L)=h(\Delta)$. Denote by $$L^{pol}\subseteq K[x_1,y_{1,1},\dots,y_{1,a_1-1},\dots,x_t,y_{t,1},\dots,y_{t,a_t-1}]$$ the polarization of $L$ and let $\Gamma$ be the simplicial complex corresponding to $L^{pol}$. By [\citen{monomial}, Corollary 1.6.3], $\Gamma$ is a CM simplicial complex and $h(\Gamma)=h(\Delta)$. Let $V_i=\{x_i,y_{i,1},\dots,y_{i,a_{i-1}}\}$ for all $1\leq i\leq t$. If $F$ is a face in $\Gamma$, then $|F\cap V_i|\leq a_i-1$, because otherwise $x_F\in I_{\Gamma}$, a contradiction. Moreover the dimension of $\Gamma$ is the dimension of the face $\bigcup_{1\leq i\leq t}V_i\setminus \{x_i\}$, so $1+\dim \Gamma=\sum_{i=1}^{t}a_i-1$. This implies that $\Gamma$ is a Cohen Macaulay $(a_1-1,\dots,a_t-1)$-balanced simplicial complex.
\end{proof}

\begin{cor}\emph{[\citen{CCV}, Corollary 2.3]}\\
The $h$-vector of CM flag simplicial complex is the $h$-vector of CM balanced simplicial complex.
\end{cor}

\begin{proof}
If $\Delta$ is a CM flag simplicial complex then $I_{\Delta}$ contains a regular sequence $f_1,\dots,f_t$ with $\deg(f_i)=2$ for all $1\leq i\leq t$ and $t$ is the height of $I_{\Delta}$. By theorem \ref{25}, there is a CM $(1,\dots,1)$-balanced (i.e. balanced) simplicial complex with the same $h$-vector.
\end{proof}

\end{document}